\documentclass[12pt]{article}
\usepackage{amsmath,amsthm,amsfonts,amssymb}
\usepackage{fullpage}
\usepackage{multirow}
\usepackage{array}
\usepackage{bbm}
\usepackage[noblocks]{authblk}
\usepackage{verbatim}
\usepackage{tikz}
\usepackage{float}
\usepackage{enumerate}
\usepackage{diagbox}
\usepackage{soul}
\usepackage{url}
\usepackage{mathtools}
\usepackage{hyperref}
\usepackage{cite}

\newtheorem{theorem}{Theorem}[section]

\newtheorem{lemma}[theorem]{Lemma}
\newtheorem{proposition}[theorem]{Proposition}

\theoremstyle{definition}
\newtheorem{definition}[theorem]{Definition}

\newlength{\Oldarrayrulewidth}

\newcommand{\Z}{\mathbb{Z}}

\newcommand{\e}{\textup{\textbf{e}}}

\begin{document}

\title{Paired $(n-1)$-to-$(n-1)$ Disjoint Path Covers in Bipartite Transposition-Like Graphs}
\author[1]{Anna~Coleman\thanks{acoleman@seu.edu}}
\author[2]{Gabrielle~Fischberg\thanks{gabrielle.fischberg@tufts.edu}}
\author[3]{Charles~Gong\thanks{charlesgong2003@gmail.com}}
\author[4]{Joshua~Harrington\thanks{joshua.harrington@cedarcrest.edu}}
\author[5]{Tony~W.~H.~Wong\thanks{wong@kutztown.edu}}
\affil[1]{Department of Mathematics, Southeastern University}
\affil[2]{Department of Mathematics, Tufts University}
\affil[3]{Department of Mathematical Sciences, Carnegie Mellon University}
\affil[4]{Department of Mathematics, Cedar Crest College}
\affil[5]{Department of Mathematics, Kutztown University of Pennsylvania}
\date{\today}

\maketitle

\begin{abstract}
A paired $k$-to-$k$ disjoint path cover of a graph $G$ is a collection of pairwise disjoint path subgraphs $P_1,P_2,\dotsc,P_k$ such that each $P_i$ has prescribed vertices $s_i$ and $t_i$ as endpoints and the union of $P_1,P_2,\dotsc,P_k$ contains all vertices of $G$. In this paper, we introduce bipartite transposition-like graphs, which are inductively constructed from lower ranked bipartite transposition-like graphs. We show that every rank $n$ bipartite transposition-like graph $G$ admit a paired $(n-1)$-to-$(n-1)$ disjoint path cover for all choices of $S=\{s_1,s_2,\dotsc,s_{n-1}\}$ and $T=\{t_1,t_2,\dotsc,t_{n-1}\}$, provided that $S$ is in one partite set of $G$ and $T$ is in the other.
\end{abstract}

\section{Introduction}\label{sec:intro}

For any positive integer $n$, let $[n]$ denote the set $\{1,2,\dotsc,n\}$. Given a graph $G$, let $S=\{s_1,s_2,\dotsc,s_k\}$ and $T=\{t_1,t_2,\dotsc,t_k\}$ be two disjoint $k$-subsets of $V(G)$. A \emph{paired $k$-to-$k$ disjoint path cover} (sometimes abbreviated as \emph{$k$-PDPC}) of $G$ is a collection of path subgraphs $P_1,P_2,\dotsc,P_k$ of $G$ such that $V(P_1),V(P_2),\dotsc,V(P_k)$ form a partition of $V(G)$ and $P_i$ has $s_i$ and $t_i$ as its endpoints for each $i\in[k]$. If $G$ admits a $k$-PDPC for every choice of disjoint $k$-subsets $S$ and $T$, then $G$ is said to be \emph{paired $k$-to-$k$ disjoint path coverable}. Note that a paired $1$-to-$1$ disjoint path coverable graph is also said to be \emph{Hamiltonian-connected}.

If $G$ is bipartite with partite sets $V_1$ and $V_2$, then $S\cup T$ must be \emph{balanced}, i.e.,
\begin{equation*}\label{eqn:balanced}
|(S\cup T)\cap V_1|-|(S\cup T)\cap V_2|=2(|V_1|-|V_2|),
\end{equation*}
in order for a $k$-PDPC to exist in $G$. In particular, if $G$ is an \emph{equitable} bipartite graph, i.e., $|V_1|=|V_2|$, then $S\cup T$ is balanced if and only if $|(S\cup T)\cap V_1|=|(S\cup T)\cap V_2|$. A bipartite graph $G$ is said to be \emph{balanced paired $k$-to-$k$ disjoint path coverable} if there is a $k$-PDPC whenever $S\cup T$ is balanced. Another name for a balanced paired $1$-to-$1$ disjoint path coverable graph is \emph{Hamiltonian-laceable}.

Due to the following proposition, research in this area often focuses on finding the largest $k$ for which a graph is paired $k$-to-$k$ disjoint path coverable or balanced paired $k$-to-$k$ disjoint path coverable.

\begin{proposition}\label{prop:differentk}\ 
%If $G$ is paired $k$-to-$k$ disjoint path coverable (respectively, balanced paired $k$-to-$k$ disjoint path coverable) for some positive integer $k$, then $G$ is paired $\widetilde{k}$-to-$\widetilde{k}$ disjoint path coverable (respectively, balanced paired $\widetilde{k}$-to-$\widetilde{k}$  disjoint path coverable) for all positive integers $\widetilde{k}\leq k$. 
\begin{enumerate}[$(a)$]
\item\label{item:differentk} Let $G$ be paired $k$-to-$k$ disjoint path coverable for some positive integer $k$. Then $G$ is paired $\ell$-to-$\ell$ disjoint path coverable for all $\ell\in[k]$. 
\item\label{item:differentkbalanced} Let $G$ be balanced paired $k$-to-$k$ disjoint path coverable for some positive integer $k$. If $G$ is equitable, then $G$ is balanced paired $\ell$-to-$\ell$  disjoint path coverable for all $\ell\in[k]$.
\item\label{item:ours} Let $G$ be an equitable bipartite graph with partite sets $V_1$ and $V_2$. If $G$ admits a $k$-PDPC for every choice of $k$-subsets $\overline{S}\subseteq V_1$ and $\overline{T}\subseteq V_2$, then for all $\ell\in[k]$, $G$ admits an $\ell$-PDPC for every choice of $\ell$-subsets $S\subseteq V_1$ and $T\subseteq V_2$.
\end{enumerate}
\end{proposition}
\begin{proof}
$(\ref{item:differentk})$ It suffices to show that $G$ is paired $(k-1)$-to-$(k-1)$ disjoint path coverable. Given two disjoint $(k-1)$-subsets $S=\{s_1,s_2,\dotsc,s_{k-1}\}$ and $T=\{t_1,t_2,\dotsc,t_{k-1}\}$ of $V(G)$, let $\overline{S}=\{\overline{s}_1,\overline{s}_2,\dotsc,\overline{s}_k\}$ and $\overline{T}=\{\overline{t}_1,\overline{t}_2,\dotsc,\overline{t}_k\}$ such that $\overline{s}_i=s_i$ and $\overline{t}_i=t_i$ for all $i\in[k-2]$, $\overline{s}_{k-1}=s_{k-1}$, $\overline{t}_k=t_{k-1}$, and $\overline{s}_k,\overline{t}_{k-1}\in V(G)\setminus(S\cup T)$ are two adjacent vertices. We note that the adjacent pair $\overline{s}_k,\overline{t}_{k-1}$ exists since $G$ is paired $k$-to-$k$ disjoint path coverable. By the assumption, $G$ admits a $k$-PDPC $\overline{P}_1,\overline{P}_2,\dotsc,\overline{P}_k$ such that $\overline{P}_i$ has endpoints $\overline{s}_i$ and $\overline{t}_i$ for each $i\in[k]$. Letting $P_i=\overline{P}_i$ for all $i\in[k-2]$ and $P_{k-1}=\overline{P}_{k-1}\cup \overline{t}_{k-1}\overline{s}_k\cup\overline{P}_k$, we obtain a desired $(k-1)$-PDPC of $G$.

%\st{\noindent$(\ref{item:differentkbalanced})$ A similar proof as in part~$(\ref{item:differentk})$ works since $|\{s_k',t_{k-1}'\}\cap V_1|=|\{s_k',t_{k-1}'\}\cap V_2|=1$, implying that $S'\cup T'$ is balanced if $S\cup T$ is balanced.}

\noindent $(\ref{item:differentkbalanced})\text{ and }(\ref{item:ours})$ A similar proof as in part~$(\ref{item:differentk})$ works by imposing the additional condition that $\overline{s}_k\in V_1$ and $\overline{t}_{k-1}\in V_2$. This ensures that $\overline{S}\cup\overline{T}$ is balanced if $S\cup T$ is balanced.
\end{proof}

Paired disjoint path covers is a subject of interest in the literature. They are particularly important in structured networks, many of which are Cayley graphs. Let $\mathcal{G}$ be a group and $\mathcal{S}\subseteq\mathcal{G}$ such that $\mathcal{S}$ is closed under inversion. An \emph{undirected Cayley graph} $\Gamma(\mathcal{G},\mathcal{S})$ has a vertex set $\mathcal{G}$, and there is an edge between $\rho,\sigma\in\mathcal{G}$ if and only if $\rho^{-1}\sigma\in\mathcal{S}$. Of particular note, if $\mathcal{G}=\Z_2^d$ with addition as the binary operation and $\mathcal{S}=\{\e_1,\e_2,\dotsc,\e_d\}$, where $\e_i$ is an ordered $d$-tuple with $1$ in the $i$-th entry and $0$ everywhere else, then $\Gamma(\mathcal{G},\mathcal{S})$ is a $d$-dimensional hypercube. More generally, if $\mathcal{G}=\Z_{m_1}\times\Z_{m_2}\times\dotsb\times\Z_{m_d}$ and $\mathcal{S}=\{\pm\e_1,\pm\e_2,\dotsc,\pm\e_d\}$, then $\Gamma(\mathcal{G},\mathcal{S})$ is a $d$-dimensional hypertorus.

Paired disjoint path covers have been studied in hypercubes \cite{ck,c2012,c2013,dg,gd,jpc-theoretcomputsci}, hypertori \cite{c2016,kw,pi,p2018}, and graphs with a similar structure \cite{jpc-informsci,kp,p2016,p2021,p2023,pkl}. It is worth mentioning that a variant topic, namely the \emph{unpaired} disjoint path covers, has also received a lot of attention \cite{c2009,kp,llmx,lmzc,ll,lw,nx,p2022,sm,zw}.

Similar research has been conducted on Cayley graphs other than hypercubes and hypertori. For example, star graphs, investigated by Park and Kim \cite{pk}, are Cayley graphs with $\mathcal{G}$ the symmetric group $S_n$ and $\mathcal{S}$ the collection of transpositions $\{(1,i):2\leq i\leq n\}$, while Li, Yang, and Meng \cite{lym} as well as Qiao, Sabir, and Meng \cite{qsm} studied Cayley graphs generated by transposition trees. If $\mathcal{G}=S_n$ and $\mathcal{T}_n=\{(i,j):1\leq i<j\leq n\}$, then the Cayley graph $\Gamma(S_n,\mathcal{T}_n)$ is called the \emph{transposition graph of rank $n$}. Note that for $n\geq2$, the transposition graph $\Gamma(S_n,\mathcal{T}_n)$ has a structural property that it is composed of $n$ subgraphs $\Gamma_1,\Gamma_2,\dotsc,\Gamma_n$, each $\Gamma_i$ is an isomorphic copy of $\Gamma(S_{n-1},\mathcal{T}_{n-1})$, and there is a perfect matching between $\Gamma_i$ and $\Gamma_j$ for all $1\leq i<j\leq n$. In view of this, we have the following two definitions.

\begin{definition}
If $G_1,G_2,\dotsc,G_n$ are $n$ graphs with the same number of vertices, then $G$ is said to be a \emph{weld} of $G_1,G_2,\dotsc,G_n$ if $G$ is composed of $G_1,G_2,\dotsc,G_n$ and there is a perfect matching between $G_i$ and $G_j$ for all $1\leq i<j\leq n$.
\end{definition}

\begin{definition}\label{def:transpositionlike}
A \emph{transposition-like graph of rank $1$} is a Hamiltonian-connected or Hamiltonian-laceable graph. For each integer $n\geq2$, a \emph{transposition-like graph of rank $n$} is a weld of $G_1,G_2,\dotsc,G_\ell$, where $\ell\geq n$ and each $G_i$ is a transposition-like graph of rank $n-1$ with the same number of vertices.
\end{definition}

In this paper, we focus on transposition-like graphs that are bipartite.  Note that these graphs are equitable when the rank is at least $2$, as shown in the following lemma.

\begin{lemma}\label{lem:bipartiteequitable}
Let $n\geq2$ be an integer. Then every bipartite transposition-like graph of rank $n$ is equitable.
\end{lemma}
\begin{proof}
If a graph is bipartite, then all its vertices can be properly colored with black and white, implying that every subgraph is bipartite. Hence, any rank $1$ transposition-like subgraph $H$ is Hamiltonian-laceable by Definition~\ref{def:transpositionlike}. If $H$ has an even number of vertices, then $H$ is equitable. If $H$ has an odd number of vertices, then we claim that the rank $2$ transposition-like subgraph $K$ containing $H$ is a weld of exactly two rank $1$ transposition-like subgraphs. To prove this claim, suppose by way of contradiction that $K$ is a weld of $H,H_2,H_3,\dotsc,H_j$ for some integer $j\geq3$. We may assume without loss of generality that there is one more black vertex in $H$ than white. Since there is a perfect matching between $H$ and $H_i$ inside $K$ for all $2\leq i\leq j$, in order to maintain the bipartite property of $K$, there must be one more white vertex in $H_i$ than black. This makes it impossible to have a perfect matching between $H_2$ and $H_3$ inside $K$, which leads to a contradiction. Therefore, $K$ is a weld of $H$ and $H_2$, where $H$ has one more black vertex than white and $H_2$ has one more white vertex than black. As a result, $K$ is equitable. Finally, our result follows from the simple observation that if $G_1,G_2,\dotsc,G_\ell$ are equitable bipartite graphs and their weld $G$ is bipartite, then $G$ is equitable. 
\end{proof}

From this lemma, we see that in order for a bipartite transposition-like graph of rank at least $2$ to have a $k$-PDPC for any positive integer $k$, it is necessary to have an equal number of black and white vertices in $S\cup T$. Here is the main theorem that we will prove in Section~\ref{sec:main}.

\begin{theorem}\label{thm:transpositionDPC}
Let $n\geq2$ be an integer, and let $G$ be a bipartite transposition-like graph of rank $n$ with partite sets $V_1$ and $V_2$. Assume that during the welding process to form $G$, every rank $1$ transposition-like graph that $G$ is built up from is either a single vertex or has an even number of vertices. Then for every choice of $(n-1)$-subsets $S\subseteq V_1$ and $T\subseteq V_2$, $G$ admits an $(n-1)$-PDPC.
\end{theorem}

%We remark that if $G$ satisfies the conditions specified in Theorem~\ref{thm:transpositionDPC}, then during the welding process to form $G$, for all $2\leq j\leq n$, every rank $j$ transposition-like graph that $G$ is built up from has an even number of vertices. 
Due to Lemma~\ref{lem:bipartiteequitable} and the structure of transposition-like graphs, the key to prove Theorem~\ref{thm:transpositionDPC} is to establish the following proposition.

\begin{proposition}\label{prop:inductionDPC}
Let $n\geq3$ be an integer, and let $G_1,G_2,\dotsc,G_\ell$ be $\ell$ graphs, where $\ell\geq n+1$. Assume that each $G_j$
\begin{enumerate}[$(a)$]
\item\label{item:cond1} has the same number of vertices with $|V(G_j)|\geq4n-2$,
\item\label{item:cond2} is bipartite and equitable with partite sets $V_{j,1}$ and $V_{j,2}$, and
\item\label{item:cond3} admits an $(n-1)$-PDPC for all $(n-1)$-subsets $S_j\subseteq V_{j,1}$ and $T_j\subseteq V_{j,2}$.
\end{enumerate}
Let $G$ be a weld of $G_1,G_2,\dotsc,G_\ell$. If $G$ is bipartite with partite sets $V_1$ and $V_2$, then for every choice of $n$-subsets $S\subseteq V_1$ and $T\subseteq V_2$, $G$ admits an $n$-PDPC.
%Let $n\geq3$ be an integer, and let $G_1,G_2,\dotsc,G_k$ be $k$ graphs, where $k\geq n$. Assume that each $G_i$
%\begin{itemize}
%\item has the same number of vertices,
%\item is bipartite with partite sets $V_{i,1}$ and $V_{i,2}$ such that $|V_{i,1}|=|V_{i,2}|$, and
%\item admits a paired $(n-2)$-to-$(n-2)$ disjoint path cover for all subsets $S_i\subseteq V_{i,1}$ and $T_i\subseteq V_{i,2}$ such that $|S_i|=|T_i|=n-2$.
%\end{itemize}
%Let $G$ be the weld of $G_1,G_2,\dotsc,G_k$. If $G$ is bipartite with partite sets $V_1$ and $V_2$ such that $|V_1|=|V_2|$, then $G$ admits a paired $(n-1)$-to-$(n-1)$ disjoint path cover for all subsets $S\subseteq V_1$ and $T\subseteq V_2$ such that $|S|=|T|=n-1$.
\end{proposition}

\section{Main Results}\label{sec:main}

From this point onward, if $G$ is bipartite with partite sets $V_1$ and $V_2$, then we color the vertices in $V_1$ black and the vertices in $V_2$ white. Since we assume that $S\subseteq V_1$ and $T\subseteq V_2$, every $s_i\in S$ is black and every $t_i\in T$ is white. For any two vertices $u$ and $v$, we use $u\sim v$ to denote a path that connects the endpoints $u$ and $v$. This notation allows $u=v$, and in such case, $u\sim v$ is simply a trivial path with one vertex. Furthermore, for any $v\in V(G)$, let $N(v)$ denote the set of neighbors of $v$ in $G$, and for any $X\subseteq V(G)$, let $N(X)$ denote the union of $N(v)$ for all $v\in X$. 

\begin{lemma}\label{lem:emptylayers}
Let $G_1,G_2,\dotsc,G_\ell$ be equitable Hamiltonian-laceable graphs with the same number of vertices, and let $G$ be a bipartite weld of $G_1,G_2,\dotsc,G_\ell$. Let $H$ be a subgraph of $G$ such that $H$ is a weld of $G_{\sigma_1},G_{\sigma_2},\dotsc,G_{\sigma_j}$ for some $\{\sigma_1,\sigma_2,\dotsc,\sigma_j\}\subseteq[\ell]$. If $H$ admits a $k$-PDPC with endpoints $S\cup T\subseteq V(H)$, then $G$ also admits a $k$-PDPC with endpoints $S\cup T$.% \st{Furthermore, each path in this $k$-PDPC of $G$ contains at least as many vertices as in the corresponding path in the $k$-PDPC of $H$.}
\end{lemma}
\begin{proof}
Without loss of generality, assume that $H$ is a weld of $G_1,G_2,\dotsc,G_j$. Let $S=\{s_1,s_2,\dotsc,s_k\}$ and $T=\{t_1,t_2,\dotsc,t_k\}$ be two disjoint $k$-subsets of $V(H)$, and let $\overline{P}_1,\overline{P}_2,\dotsc,\overline{P}_k$ be a $k$-PDPC of $H$. In particular, let $\overline{P}_1=s_1\sim u_jv_j\sim t_1$, where $u_j$ or $v_j$ are not required to be distinct from $s_1$ or $t_1$. Note that $u_j$ and $v_j$ are of opposite color, so the neighbors of $u_j$ and $v_j$ are also of opposite color. For each $j\leq i\leq\ell-1$, let $u_i^*,v_i^*\in V(G_{i+1})$ be the neighbors of $u_i$ and $v_i$, respectively, and let $\overline{P}_{i+1}=u_i^*\sim u_{i+1}v_{i+1}\sim v_i^*$ be a Hamiltonian path of $G_{i+1}$ that connects $u_i^*$ and $v_i^*$, where $u_{i+1}$ and $v_{i+1}$ are not required to be distinct from $u_i^*$ or $v_i^*$. Let
$$P_1=\underset{\text{from }\overline{P}_1}{\underbrace{s_1\sim u_j}}\underset{\text{from }\overline{P}_{j+1}}{\underbrace{u_j^*\sim u_{j+1}}}\underset{\text{from }\overline{P}_{j+2}}{\underbrace{u_{j+1}^*\sim u_{j+2}}}\;\dotsb\;\underset{\text{from }\overline{P}_\ell}{\underbrace{u_{\ell-1}^*\sim u_\ell}}\underset{\text{from }\overline{P}_\ell}{\underbrace{v_\ell\sim v_{\ell-1}^*}}\;\dotsb\;\underset{\text{from }\overline{P}_{j+1}}{\underbrace{v_{j+1}\sim v_j^*}}\underset{\text{from }\overline{P}_1}{\underbrace{v_j\sim t_1}}$$
and $P_i=\overline{P}_i$ for $2\leq i\leq k$. Then $P_1,P_2,\dotsc,P_k$ is a $k$-PDPC of $G$.%\st{, and each path contains at least as many vertices as in $P_1,P_2,\dotsc,P_k$, respectively}.
\end{proof}

Next, we are going to prove Proposition~\ref{prop:inductionDPC}, which provides the inductive step for the proof of Theorem~\ref{thm:transpositionDPC}.

\begin{proof}[Proof of Proposition~$\ref{prop:inductionDPC}$]
Let $S=\{s_1,s_2,\dotsc,s_n\}\subseteq V_1$ and $T=\{t_1,t_2,\dotsc,t_n\}\subseteq V_2$ be $n$-subsets of $V(G)$. For each $j\in[\ell]$, let $S_j=V(G_j)\cap S$, $T_j=V(G_j)\cap T$, $S_j'=\{s_i\in S_j:t_i\notin T_j\}$, $S_j''=\{s_i\in S_j:t_i\in T_j\}$, $T_j'=\{t_i\in T_j:s_i\notin S_j\}$, $T_j''=\{t_i\in T_j:s_i\in S_j\}$, and $w_j=|\{i\in[n]:V(G_j)\cap\{s_i,t_i\}\neq\emptyset\}|$. Note that $|S_j''|=|T_j''|$ and $|S_j|+|T_j'|=|S_j'|+|T_j|=w_j$. %$w_s,w_t,w:[\ell]\to[n]$ such that $w_s(j)=|V(G_j)\cap S|$, $w_t(j)=|V(G_j)\cap T|$, and $w(j)=|\{i\in[n]:V(G_j)\cap\{s_i,t_i\}\neq\emptyset\}|$ for all $j\in[\ell]$. 
Furthermore, let $\tau:V(G)\to[\ell]$ such that $v\in V(G_{\tau(v)})$ for all $v\in V(G)$. Due to condition~$(\ref{item:cond1})$ on $G_j$ for all $j\in[\ell]$, there are at least $4n-2$ vertices in $V(G_j)$; by condition~$(\ref{item:cond2})$, there are at least $2n-1$ black and $2n-1$ white vertices in $V(G_j)$. 

\begin{enumerate}[\textit{Case} $1$:]
\item\label{item:wj<=n-1} $w_j\leq n-1$ for all $j\in[\ell]$.

For each $i\in[n]$ such that $\tau(s_i)\neq\tau(t_i)$, we define $v_i,u_i\in V(G)$ inductively. Initially, let $\mathcal{V}=\mathcal{U}=\emptyset$. Let $j=\tau(s_i)$ and $j^*=\tau(t_i)$. We can define $v_i$ to be a white vertex in $V(G_j)\setminus\big(T_j\cup(\mathcal{V}\cap V(G_j))\cup N(S_{j^*}\cup(\mathcal{U}\cap V(G_{j^*}))\big)$ since
\begin{align*}
&\;\big|T_j\cup(\mathcal{V}\cap V(G_j))\cup\big(N(S_{j^*}\cup(\mathcal{U}\cap V(G_{j^*}))\cap V(G_j)\big)\big|\\
\leq&\;|T_j|+|\mathcal{V}\cap V(G_j)|+|S_{j^*}|+|\mathcal{U}\cap V(G_{j^*})|\\
\leq&\;|T_j|+|S_j'|-1+|S_{j^*}|+|T_{j^*}'|-1\\
=&\;w_j+w_{j^*}-2\\
\leq&\;2(n-1)-2\\
<&\;2n-1.
\end{align*}
Let $u_i\in V(G_{j^*})$ be the unique neighbor of $v_i$, and we update $\mathcal{V}:=\mathcal{V}\cup\{v_i\}$ and $\mathcal{U}:=\mathcal{U}\cup\{u_i\}$.

By condition~$(\ref{item:cond3})$ on each $G_j$ and Proposition~\ref{prop:differentk}$(\ref{item:ours})$, there is a $w_j$-PDPC in $G_j$ with endpoints $S_j'\cup\{u_i:t_i\in T_j'\}\cup S_j''$ and $\{v_i:s_i\in S_j'\}\cup T_j'\cup T_j''$, where paths are of the form $s_i\sim v_i$, $u_i\sim t_i$, or $s_i\sim t_i$. For each $i\in[n]$, let
$$P_i=\underset{G_{\tau(s_i)}}{\underbrace{s_i\sim v_i}}\underset{G_{\tau(t_i)}}{\underbrace{u_i\sim t_i}}$$
if $s_i\in S_{\tau(s_i)}'$ and let
$$P_i=\underset{G_{\tau(s_i)}}{\underbrace{s_i\sim t_i}}$$
if $s_i\in S_{\tau(s_i)}''$. Then $P_1,P_2,\dotsc,P_n$ form an $n$-PDPC of the weld of the graphs in $\{G_j:j\in[\ell]\text{ and }w_j>0\}$, and we obtain a desired $n$-PDPC of $G$ by Lemma~\ref{lem:emptylayers}.

%For each $j\in[\ell]$, note that $T_j'\cup T_j''$ contains at most $n-1-|S_j'|$ white vertices in $G_j$, while $N(S\setminus(S_j'\cup S_j''))\cap V(G_j)$ contains at most $n-|S_j'\cup S_j''|$ white vertices in $G_j$. Hence, for each $s_i\in S_j'$, we can define a distinct white vertex $v_i\in V(G_j)\setminus\big(T_j'\cup T_j''\cup(N(S\setminus(S_j'\cup S_j''))\cap V(G_j))\big)$ since $|S_j'|\leq2(n-1)-(n-1-|S_j'|)-(n-|S_j'\cup S_j''|)$. Next, for each $v_i$ defined, let $u_i$ be the unique black vertex in $N(v_i)\cap V(G_{\tau(t_i)})$. Note that $u_i\notin S_{\tau(t_i)}'\cup S_{\tau(t_i)}''$ due to the definition of $v_i$. %Similarly, for each $t_i\in T_j'$, we can define a distinct black vertex $u_i\in V(G_j)\setminus\big(S_j'\cup S_j''\cup(N(T\setminus(T_j'\cup T_j''))\cap V(G_j))\big)$. 
%By condition~$(\ref{item:cond3})$ on $G_j$, there is an $(n-1)$-PDPC in $G_j$ with endpoints $S_j'\cup\{u_i:t_i\in T_j'\}\cup S_j''$ and $\{v_i:s_i\in S_j'\}\cup T_j'\cup T_j''$.\\\\\\
%For each $j\in[\ell]$, there are at least $n-1$ black and $n-1$ white vertices in $V(G_j)\setminus(S\cup T)$. For each $i\in[n]$, if $\tau(s_i)\neq\tau(t_i)$, then let $v_i\in V(G_{\tau(s_i)})\setminus(S\cup T\cup\{v_1,v_2,\dotsc,v_{i-1}\})$ be a white vertex and $u_i\in V(G_{\tau(t_i)})$ be a neighbor of $v_i$ such that 
%Let $I=\{i\in[n]:\{s_i,t_i\}\subseteq V(G_j)\text{ for some }j\in[\ell]\}$, and let $I^\mathcal{C}=[n]\setminus I$. For each $i\in I^\mathcal{C}$, if let $j_{i1},j_{i2}\in[\ell]$ such that $s_i\in V(G_{j_{i1}})$ and $t_i\in V(G_{j_{i2}})$.

\item\label{item:STinVGj} There exists $j\in[\ell]$ such that $S\cup T\subseteq V(G_j)$.

Without loss of generality, assume that $S\cup T\subseteq V(G_1)$. By condition~$(\ref{item:cond3})$ on $G_1$, there is an $(n-1)$-PDPC $\overline{P}_1,\overline{P}_2,\dotsc,\overline{P}_{n-1}$ in $G_1$ with endpoints $\{s_1,s_2,\dotsc,s_{n-1}\}$ and $\{t_1,t_2,\dotsc,t_{n-1}\}$.

If there exists $i\in[n-1]$ such that the path $s_i\sim t_i$ contains both $s_n$ and $t_n$, then without loss of generality, assume that $\overline{P}_1$ is either $s_1\sim v_1s_n\sim t_nu_1\sim t_1$ or $s_1\sim u_1t_n\sim s_nv_1\sim t_1$, where $u_1$ is black, $v_1$ is white, and $u_1,v_1$ are not required to be distinct from $s_1,t_1$, respectively. Let $u_1^*,v_1^*\in V(G_2)$ be the unique neighbors of $u_1,v_1$, respectively. By condition~$(\ref{item:cond3})$ on $G_2$ and Proposition~\ref{prop:differentk}$(\ref{item:ours})$, there is a Hamiltonian path in $G_2$ with endpoints $v_1^*$ and $u_1^*$. Let $P_1$ be either
$$\underset{\overline{P}_1}{\underbrace{s_1\sim v_1}}\underset{G_2}{\underbrace{v_1^*\sim u_1^*}}\underset{\overline{P}_1}{\underbrace{u_1\sim t_1}}\quad\text{or}\quad\underset{\overline{P}_1}{\underbrace{s_1\sim u_1}}\underset{G_2}{\underbrace{u_1^*\sim v_1^*}}\underset{\overline{P}_1}{\underbrace{v_1\sim t_1}},$$
$P_i=\overline{P}_i$ for $2\leq i\leq n-1$, and
$$P_n=\underset{\overline{P}_1}{\underbrace{s_n\sim t_n}}.$$
Then $P_1,P_2,\dotsc,P_n$ form an $n$-PDPC of the weld of $G_1$ and $G_2$, and we obtain a desired $n$-PDPC of $G$ by Lemma~\ref{lem:emptylayers}.

If there does not exist $i\in[n-1]$ such that the path $s_i\sim t_i$ contains both $s_n$ and $t_n$, then without loss of generality, assume that $\overline{P}_1=s_1\sim u_1v_ns_nv_1\sim t_1$ and $\overline{P}_2=s_2\sim u_2t_nu_nv_2\sim t_2$, where $u_1,u_2,u_n$ are black, $v_1,v_2,v_n$ are white, and $u_1,v_1,u_2,v_2$ are not required to be distinct from $s_1,t_1,s_2,t_2$, respectively. Let $u_1^*,u_2^*,v_1^*,v_2^*\in V(G_2)$ and $u_n^*,v_n^*\in V(G_3)$ be the unique neighbors of $u_1,u_2,v_1,v_2,u_n,v_n$, respectively. By condition~$(\ref{item:cond3})$ on $G_2$ and $G_3$ and Proposition~\ref{prop:differentk}$(\ref{item:ours})$, there is a $2$-PDPC in $G_2$ with endpoints $\{v_1^*,v_2^*\}$ and $\{u_1^*,u_2^*\}$ and a Hamiltonian path in $G_3$ with endpoints $v_n^*$ and $u_n^*$. Let 
$$P_1=\underset{\overline{P}_1}{\underbrace{s_1\sim u_1}}\underset{G_2}{\underbrace{u_1^*\sim v_1^*}}\underset{\overline{P}_1}{\underbrace{v_1\sim t_1}},$$
$$P_2=\underset{\overline{P}_2}{\underbrace{s_2\sim u_2}}\underset{G_2}{\underbrace{u_2^*\sim v_2^*}}\underset{\overline{P}_2}{\underbrace{v_2\sim t_2}},$$
$P_i=\overline{P}_i$ for $3\leq i\leq n-1$, and 
$$P_n=\underset{\overline{P}_1}{\underbrace{s_nv_n}}\underset{G_3}{\underbrace{v_n^*\sim u_n^*}}\underset{\overline{P}_2}{\underbrace{u_nt_n}}.$$
Then $P_1,P_2,\dotsc,P_n$ form an $n$-PDPC of the weld of $G_1$, $G_2$, and $G_3$, and we obtain a desired $n$-PDPC of $G$ by Lemma~\ref{lem:emptylayers}.

\item\label{item:SinVG1TinVG2} There exist $j_1,j_2\in[\ell]$, where $j_1\neq j_2$, such that $S\subseteq V(G_{j_1})$ and $T\subseteq V(G_{j_2})$.

Without loss of generality, assume that $S\subseteq V(G_1)$ and $T\subseteq V(G_2)$. Let $v_1,v_2,\dotsc,v_{n-1}\in V(G_1)$ be distinct white vertices. By condition~$(\ref{item:cond3})$ on $G_1$, there is an $(n-1)$-PDPC $\overline{P}_1,\overline{P}_2,\dotsc,\overline{P}_{n-1}$ in $G_1$ with endpoints $\{s_1,s_2,\dotsc,s_{n-1}\}$ and $\{v_1,v_2,\dotsc,v_{n-1}\}$. Without loss of generality, assume that $\overline{P}_1$ contains $s_n$, i.e., $\overline{P}_1=s_1\sim v_ns_n\sim v_1$, where $v_n$ is white. Rename the vertices $v_1$ and $v_n$ as $v_n$ and $v_1$, respectively, so $\overline{P}_1=s_1\sim v_1s_n\sim v_n$. Let $v_1^*,v_2^*,\dotsc,v_{n-1}^*\in V(G_2)$ be the unique neighbors of $v_1,v_2,\dotsc,v_{n-1}$, respectively. By condition~$(\ref{item:cond3})$ on $G_2$, there is an $(n-1)$-PDPC $\widehat{P}_1,\widehat{P}_2,\dotsc,\widehat{P}_{n-1}$ in $G_2$ with endpoints $\{v_1^*,v_2^*,\dotsc,v_{n-1}^*\}$ and $\{t_1,t_2,\dotsc,t_{n-1}\}$. Again without loss of generality, assume that $\widehat{P}_1$ contains $t_n$, i.e., $\widehat{P}_1=v_1^*\sim u_0t_nu_nv_0\sim t_1$, where $u_0$ and $u_n$ are black, $v_0$ is white, and $u_0,v_0$ are not required to be distinct from $v_1^*,t_1$, respectively. Let $u_0^*,v_0^*\in V(G_3)$ and $u_n^*,v_n^*\in V(G_4)$ be the unique neighbors of $u_0,v_0,u_n,v_n$, respectively. By condition~$(\ref{item:cond3})$ on $G_3$ and $G_4$ together with Proposition~\ref{prop:differentk}$(\ref{item:ours})$, there is a Hamiltonian path of $G_3$ with endpoints $v_0^*$ and $u_0^*$ and a Hamiltonian path of $G_4$ with endpoints $v_n^*$ and $u_n^*$. Let 
$$P_1=\underset{\overline{P}_1}{\underbrace{s_1\sim v_1}}\underset{\widehat{P}_1}{\underbrace{v_1^*\sim u_0}}\underset{G_3}{\underbrace{u_0^*\sim v_0^*}}\underset{\widehat{P}_1}{\underbrace{v_0\sim t_1}},$$
$$P_i=\underset{\overline{P}_i}{\underbrace{s_i\sim v_i}}\underset{\widehat{P}_i}{\underbrace{v_i^*\sim t_i}}$$
for $2\leq i\leq n-1$, and 
$$P_n=\underset{\overline{P}_1}{\underbrace{s_n\sim v_n}}\underset{G_4}{\underbrace{v_n^*\sim u_n^*}}\underset{\widehat{P}_1}{\underbrace{u_nt_n}}.$$
Then $P_1,P_2,\dotsc,P_n$ form an $n$-PDPC of the weld of $G_1,G_2,G_3,G_4$, and we obtain a desired $n$-PDPC of $G$ by Lemma~\ref{lem:emptylayers}.

\item\label{item:SinVGj0TnotinVGj0} There exists $j_0\in[\ell]$ such that $S\subseteq V(G_{j_0})$ or $T\subseteq V(G_{j_0})$ but not both, and $w_j\leq n-1$ for all $j\in[\ell]\setminus\{j_0\}$.

Without loss of generality, assume that $S\subseteq V(G_1)$ and $t_n\notin V(G_1)$. Note that $T_1=T_1''$ and $T_j=T_j'$ for $2\leq j\leq\ell$. For every $s_i\in S_1'\setminus\{s_n\}$, let $v_i\in V(G_1)$ be a white vertex such that $\{v_i:s_i\in S_1'\setminus\{s_n\}\}\cup T_1$ is an $(n-1)$-subset of $V(G_1)$. By condition~$(\ref{item:cond3})$ on $G_1$, there is an $(n-1)$-PDPC $\overline{P}_1,\overline{P}_2,\dotsc,\overline{P}_{n-1}$ in $G_1$ with endpoints $\{s_1,s_2,\dotsc,s_{n-1}\}$ and $\{v_i:s_i\in S_1'\setminus\{s_n\}\}\cup T_1$, where paths are of the form $s_i\sim v_i$ or $s_i\sim t_i$. Without loss of generality, assume that the path $\overline{P}_1$ contains $s_n$.

If $\overline{P}_1=s_1\sim v_ns_n\sim v_1$ where $v_n$ is white, then rename the vertices $v_1$ and $v_n$ as $v_n$ and $v_1$, respectively, so $\overline{P}_1=s_1\sim v_1s_n\sim v_n$. For each $s_i\in S_1'$, let $v_i^*\in V(G_{\tau(t_i)})$ be the unique neighbor of $v_i$. By condition~$(\ref{item:cond3})$ on $G_j$ for $2\leq j\leq\ell$ and Proposition~\ref{prop:differentk}$(\ref{item:ours})$, there is a $w_j$-PDPC in $G_j$ with endpoints $\{v_i^*:t_i\in T_j\}$ and $T_j$. Let 
$$P_i=\underset{G_1}{\underbrace{s_i\sim v_i}}\underset{G_{\tau(t_i)}}{\underbrace{v_i^*\sim t_i}}$$
for each $s_i\in S_1'$ and $P_i=\overline{P}_i$ for each $s_i\in S_1''$. Then $P_1,P_2,\dotsc,P_n$ form an $n$-PDPC of the weld of the graphs in $\{G_j:j\in[\ell]\text{ and }w_j>0\}$, and we obtain a desired $n$-PDPC of $G$ by Lemma~\ref{lem:emptylayers}.

Otherwise, $\overline{P}_1=s_1\sim u_0v_ns_nv_0\sim t_1$, where $u_0$ is black, $v_0$ and $v_n$ are white, and $u_0,v_0$ are not required to be distinct from $s_1,t_1$, respectively. For each $s_i\in S_1'$, let $v_i^*\in V(G_{\tau(t_i)})$ be the unique neighbor of $v_i$. Since $|T\setminus T_1|<n\leq\ell-1$, we can assume without loss of generality that $T\cap V(G_\ell)=\emptyset$. Let $u_0^*,v_0^*\in V(G_\ell)$ be the unique neighbors of $u_0$ and $v_0$, respectively. By condition~$(\ref{item:cond3})$ on $G_\ell$ and Proposition~\ref{prop:differentk}$(\ref{item:ours})$, there is a Hamiltonian path in $G_\ell$ with endpoints $v_0^*$ and $u_0^*$. For each $s_i\in S_1'$, let $v_i^*\in V(G_{\tau(t_i)})$ be the unique neighbor of $v_i$. By condition~$(\ref{item:cond3})$ on $G_j$ for $2\leq j\leq\ell$ and Proposition~\ref{prop:differentk}$(\ref{item:ours})$, there is a $w_j$-PDPC in $G_j$ with endpoints $\{v_i^*:t_i\in T_j\}$ and $T_j$. Let 
$$P_1=\underset{\overline{P}_1}{\underbrace{s_1\sim u_0}}\underset{G_\ell}{\underbrace{u_0^*\sim v_0^*}}\underset{\overline{P}_1}{\underbrace{v_0\sim t_1}},$$
$$P_i=\underset{\overline{P}_i}{\underbrace{s_i\sim v_i}}\underset{G_{\tau(t_i)}}{\underbrace{v_i^*\sim t_i}}$$
for each $s_i\in S_1'\setminus\{s_n\}$, $P_i=\overline{P}_i$ for each $s_i\in S_1''\setminus\{s_1\}$, and
$$P_n=\underset{\overline{P}_1}{\underbrace{s_nv_n}}\underset{G_{\tau(t_n)}}{\underbrace{v_n^*\sim t_n}}.$$
Then $P_1,P_2,\dotsc,P_n$ form an $n$-PDPC of the weld of the graphs in $\{G_j:j\in[\ell]\text{ and }w_j>0\}\cup\{G_\ell\}$, and we obtain a desired $n$-PDPC of $G$ by Lemma~\ref{lem:emptylayers}.

\item There exists $j_0\in[\ell]$ such that $w_{j_0}=n$, $S\nsubseteq V(G_{j_0})$, $T\nsubseteq V(G_{j_0})$, and $w_j\leq n-1$ for all $j\in[\ell]\setminus\{j_0\}$.

Without loss of generality, assume that $w_1=n$. Since $|S_1|+|T_1|\geq n\geq3$, we have $\max(|S_1|,|T_1|)\geq2$. Without loss of generality, assume that $|T_1|\geq2$ and $s_n\in S_1'$. Note that $S_j=S_j'$ and $T_j=T_j'$ for $2\leq j\leq\ell$. Similar to \textit{Case}~$\ref{item:wj<=n-1}$, for each $i\in[n-1]$ such that $\tau(s_i)\neq\tau(t_i)$, we define $v_i,u_i\in V(G)$ inductively. Initially, let $\mathcal{V}=\mathcal{U}=\emptyset$. Let $j=\tau(s_i)$ and $j^*=\tau(t_i)$, where $1\in\{j,j^*\}$. We can define $v_i$ to be a white vertex in $V(G_j)\setminus\big(T_j\cup(\mathcal{V}\cap V(G_j))\cup N(S_{j^*}\cup(\mathcal{U}\cap V(G_{j^*}))\big)$ since
\begin{align*}
&\;\big|T_j\cup(\mathcal{V}\cap V(G_j))\cup\big(N(S_{j^*}\cup(\mathcal{U}\cap V(G_{j^*}))\cap V(G_j)\big)\big|\\
\leq&\;|T_j|+|\mathcal{V}\cap V(G_j)|+|S_{j^*}|+|\mathcal{U}\cap V(G_{j^*})|\\
\leq&\;\max(|T_1|+|S_1'|-2+|S_{j^*}|+|T_{j^*}|-1,\;|T_j|+|S_j|-1+|S_1|+|T_1'|-1)\\
%\leq&\;\max(|T_j|+|S_j'|-2+|S_{j^*}|+|T_{j^*}'|-1,\;|T_j|+|S_j'|-1+|S_{j^*}|+|T_{j^*}'|-2)\\
=&\;w_j+w_{j^*}-2\\
\leq&\;2n-3\\
<&\;2n-1.
\end{align*}
Let $u_i\in V(G_{j^*})$ be the unique neighbor of $v_i$, and we update $\mathcal{V}:=\mathcal{V}\cup\{v_i\}$ and $\mathcal{U}:=\mathcal{U}\cup\{u_i\}$.

By condition~$(\ref{item:cond3})$ on $G_1$, there is an $(n-1)$-PDPC $\overline{P}_1,\overline{P}_2,\dotsc,\overline{P}_{n-1}$ in $G_1$ with endpoints $(S_1'\setminus\{s_n\})\cup\{u_i:t_i\in T_1'\}\cup S_1''$ and $\{v_i:s_i\in S_1'\setminus\{s_n\}\}\cup T_1'\cup T_1''$, where paths are of the form $s_i\sim v_i$, $u_i\sim t_i$, or $s_i\sim t_i$. Without loss of generality, assume that $\overline{P}_1$ contains $s_n$. Then $\overline{P}_1$ is either $s_1\sim u_0v_ns_nv_0\sim v_1$, $u_1\sim u_0v_ns_nv_0\sim t_1$, or $s_1\sim u_0v_ns_nv_0\sim t_1$, where $u_0$ is black, $v_0$ and $v_n$ are white, $u_0$ is not required to be distinct from $s_1$ or $u_1$, and $v_0$ is not required to be distinct from $v_1$ or $t_1$.

If there exists $j\in[\ell]$ such that $w_j=0$, then assume without loss of generality that $w_\ell=0$. Let $u_0^*,v_0^*,v_n^*\in V(G_\ell)$ be the unique neighbors of $u_0,v_0,v_n$, respectively. Furthermore, we define $u_n$ to be a black vertex in $V(G_{\tau(t_n)})\setminus\big(S_{\tau(t_n)}\cup(\mathcal{U}\cap V(G_{\tau(t_n)}))\cup N(u_0^*)\big)$. Let $u_n^*\in V(G_\ell)$ be the unique neighbor of $u_n$. By condition~$(\ref{item:cond3})$ on each $G_j$ and Proposition~\ref{prop:differentk}$(\ref{item:ours})$, there is a $w_j$-PDPC in $G_j$ for $2\leq j\leq\ell-1$ with endpoints $S_j\cup\{u_i:t_i\in T_j\}$ and $\{v_i:s_i\in S_j\}\cup T_j$, where paths are of the form $s_i\sim v_i$ or $u_i\sim t_i$, and there is a $2$-PDPC in $G_\ell$ with endpoints $\{v_0^*,v_n^*\}$ and $\{u_0^*,u_n^*\}$. Let $P_1$ be either
$$\underset{\overline{P}_1}{\underbrace{s_1\sim u_0}}\underset{G_\ell}{\underbrace{u_0^*\sim v_0^*}}\underset{\overline{P}_1}{\underbrace{v_0\sim v_1}}\underset{G_{\tau(t_1)}}{\underbrace{u_1\sim t_1}},$$
$$\underset{G_{\tau(s_1)}}{\underbrace{s_1\sim v_1}}\underset{\overline{P}_1}{\underbrace{u_1\sim u_0}}\underset{G_\ell}{\underbrace{u_0^*\sim v_0^*}}\underset{\overline{P}_1}{\underbrace{v_0\sim t_1}},$$
or
$$\underset{\overline{P}_1}{\underbrace{s_1\sim u_0}}\underset{G_\ell}{\underbrace{u_0^*\sim v_0^*}}\underset{\overline{P}_1}{\underbrace{v_0\sim t_1}}.$$
Let
$$P_i=\underset{G_{\tau(s_i)}}{\underbrace{s_i\sim v_i}}\underset{G_{\tau(t_i)}}{\underbrace{u_i\sim t_i}}$$
if $s_i\in S_1'\setminus\{s_n\}$ or $t_i\in T_1'$, $P_i=\overline{P}_i$ if $s_i\in S_1''$, and
$$P_n=\underset{\overline{P}_1}{\underbrace{s_nv_n}}\underset{G_\ell}{\underbrace{v_n^*\sim u_n^*}}\underset{G_{\tau(t_n)}}{\underbrace{u_n\sim t_n}}.$$
Then $P_1,P_2,\dotsc,P_n$ form an $n$-PDPC of the weld of the graphs in $\{G_j:j\in[\ell]\text{ and }w_j>0\}\cup\{G_\ell\}$, and we obtain a desired $n$-PDPC of $G$ by Lemma~\ref{lem:emptylayers}.

If $w_j>0$ for all $j\in[\ell]$, then since $|(S\cup T)\setminus(S_1\cup T_1)|\leq n\leq\ell-1$, we know that $G_j$ contains exactly one vertex from $(S\cup T)\setminus(S_1\cup T_1)$ for each $2\leq j\leq\ell$. Furthermore, since $|T_1|\geq2$, there exist $2\leq j_1<j_2\leq\ell$ such that $T_{j_1}=T_{j_2}=\emptyset$, so we may assume without loss of generality that $T_2=\emptyset$ and $S_2=\{s_2\}$. Let $u_0^*\in V(G_2)$ and $v_0^*,v_n^*\in V(G_{\tau(t_n)})$ be the unique neighbors of $u_0,v_0,v_n$, respectively. Note that $u_0^*\neq v_2$ since $v_2$ is a neighbor of $u_2\in G_1$ and $u_2\neq u_0$. We define $x_0$ to be a black vertex in $V(G_2)\setminus\big(\{s_2\}\cup N(t_n)\big)$. Let $y_0\in V(G_{\tau(t_n)})$ be the unique neighbor of $x_0$. By condition~$(\ref{item:cond3})$ on each $G_j$ and Proposition~\ref{prop:differentk}$(\ref{item:ours})$, there is a $2$-PDPC in $G_2$ with endpoints $\{s_2,x_0\}$ and $\{v_2,u_0^*\}$, a $2$-PDPC in $G_{\tau(t_n)}$ with endpoints $\{v_0^*,v_n^*\}$ and $\{y_0,t_n\}$, and a Hamiltonian path in $G_j$ with endpoints $s_i$ and $v_i$ or endpoints $u_i$ and $t_i$ for $j\in[\ell]\setminus\{1,2,\tau(t_n)\}$. Let $P_1$ be either
$$\underset{\overline{P}_1}{\underbrace{s_1\sim u_0}}\underset{G_2}{\underbrace{u_0^*\sim x_0}}\underset{G_{\tau(t_n)}}{\underbrace{y_0\sim v_0^*}}\underset{\overline{P}_1}{\underbrace{v_0\sim v_1}}\underset{G_{\tau(t_1)}}{\underbrace{u_1\sim t_1}},$$
$$\underset{G_{\tau(s_1)}}{\underbrace{s_1\sim v_1}}\underset{\overline{P}_1}{\underbrace{u_1\sim u_0}}\underset{G_2}{\underbrace{u_0^*\sim x_0}}\underset{G_{\tau(t_n)}}{\underbrace{y_0\sim v_0^*}}\underset{\overline{P}_1}{\underbrace{v_0\sim t_1}},$$
or
$$\underset{\overline{P}_1}{\underbrace{s_1\sim u_0}}\underset{G_2}{\underbrace{u_0^*\sim x_0}}\underset{G_{\tau(t_n)}}{\underbrace{y_0\sim v_0^*}}\underset{\overline{P}_1}{\underbrace{v_0\sim t_1}}.$$
Let
$$P_i=\underset{G_{\tau(s_i)}}{\underbrace{s_i\sim v_i}}\underset{G_{\tau(t_i)}}{\underbrace{u_i\sim t_i}}$$
for $2\leq i\leq n-1$ and
$$P_n=\underset{\overline{P}_1}{\underbrace{s_nv_n}}\underset{G_{\tau(t_n)}}{\underbrace{v_n^*\sim t_n}}.$$
Then $P_1,P_2,\dotsc,P_n$ form an $n$-PDPC of $G$.

\item There exist $j_1,j_2\in[\ell]$, where $j_1\neq j_2$, such that $w_{j_1}=w_{j_2}=n$ and $S\nsubseteq V(G_{j_1})$.

Without loss of generality, assume that $w_1=w_2=n$ and $s_n\in S_1'$. Note that $S_j=S_j'$ and $T_j=T_j'$ for $j\in\{1,2\}$. Similar to \textit{Case}~$\ref{item:wj<=n-1}$, for each $i\in[n-1]$, we define $v_i,u_i\in V(G)$ inductively. Initially, let $\mathcal{V}=\mathcal{U}=\emptyset$. Let $j=\tau(s_i)$ and $j^*=\tau(t_i)$, where $\{j,j^*\}=\{1,2\}$. Let $v_i$ be a white vertex in $V(G_j)\setminus\big(T_j\cup(\mathcal{V}\cap V(G_j))\cup N(S_{j^*}\cup(\mathcal{U}\cap V(G_{j^*}))\big)$ and let $u_i\in V(G_{j^*})$ be the unique neighbor of $v_i$. Note that $\mathcal{V}\cap V(G_j)=N(\mathcal{U}\cap V(G_{j^*}))\cap V(G_j)$, so our definition of $v_i$ and $u_i$ is valid since
\begin{align*}
&\;\big|T_j\cup(\mathcal{V}\cap V(G_j))\cup\big(N(S_{j^*}\cup(\mathcal{U}\cap V(G_{j^*}))\cap V(G_j)\big)\big|\\
\leq&\;|T_j|+|\mathcal{V}\cap V(G_j)|+|S_{j^*}|\\
\leq&\;\max(|T_1|+|S_1|-2+|S_2|,\;|T_2|+|S_2|-1+|S_1|)\\
\leq&\;\max(n-2+n-1,\;n-1+n-1)\\
=&\;2n-2\\
<&\;2n-1.
\end{align*}
After each $v_i$ and $u_i$ are defined, we update $\mathcal{V}:=\mathcal{V}\cup\{v_i\}$ and $\mathcal{U}:=\mathcal{U}\cup\{u_i\}$.

By condition~$(\ref{item:cond3})$ on $G_1$, there is an $(n-1)$-PDPC $\overline{P}_1,\overline{P}_2,\dotsc,\overline{P}_{n-1}$ in $G_1$ with endpoints $(S_1\setminus\{s_n\})\cup\{u_i:t_i\in T_1\}$ and $\{v_i:s_i\in S_1\setminus\{s_n\}\}\cup T_1$, where paths are of the form $s_i\sim v_i$ or $u_i\sim t_i$. Without loss of generality, assume that $\overline{P}_1$ contains $s_n$.

If $\overline{P}_1=s_1\sim v_ns_n\sim v_1$ where $v_n$ is white, then rename the vertices $v_1,v_n,u_1$ as $v_n,v_1,u_n$, respectively, so $\overline{P}_1=s_1\sim v_1s_n\sim v_n$. By condition~$(\ref{item:cond3})$ on $G_2$, there is an $(n-1)$-PDPC $\widehat{P}_2,\widehat{P}_3,\dotsc,\widehat{P}_n$ in $G_2$ with endpoints $S_2\cup\{u_i:t_i\in T_2\setminus\{t_1\}\}$ and $\{v_i:s_i\in S_2\}\cup(T_2\setminus\{t_1\})$, where paths are of the form $s_i\sim v_i$ or $u_i\sim t_i$. Without loss of generality, assume that $\widehat{P}_2$ contains $t_1$. Let $\widehat{P}_2$ be either $s_2\sim x_2t_1u_1y_2\sim v_2$ or $u_2\sim x_2t_1u_1y_2\sim t_2$, where $u_1$ and $x_2$ are black, $y_2$ is white, $x_2$ is not required to be distinct from $s_2$ or $u_2$, and $y_2$ is not required to be distinct from $v_2$ or $t_2$. Let $u_1^*,v_1^*\in V(G_3)$ and $x_2^*,y_2^*\in V(G_4)$ be the unique neighbors of $u_1,v_1,x_2,y_2$, respectively. By condition~$(\ref{item:cond3})$ on $G_3$ and $G_4$ together with Proposition~\ref{prop:differentk}$(\ref{item:ours})$, there is a Hamiltonian path of $G_3$ with endpoints $v_1^*$ and $u_1^*$ and a Hamiltonian path of $G_4$ with endpoints $y_2^*$ and $x_2^*$. Let
$$P_1=\underset{\overline{P}_1}{\underbrace{s_1\sim v_1}}\underset{G_3}{\underbrace{v_1^*\sim u_1^*}}\underset{\widehat{P}_2}{\underbrace{u_1t_1}},$$
$P_2$ be either 
$$\underset{\widehat{P}_2}{\underbrace{s_2\sim x_2}}\underset{G_4}{\underbrace{x_2^*\sim y_2^*}}\underset{\widehat{P}_2}{\underbrace{y_2\sim v_2}}\underset{\overline{P}_2}{\underbrace{u_2\sim t_2}}\quad\text{or}\quad\underset{\overline{P}_2}{\underbrace{s_2\sim v_2}}\underset{\widehat{P}_2}{\underbrace{u_2\sim x_2}}\underset{G_4}{\underbrace{x_2^*\sim y_2^*}}\underset{\widehat{P}_2}{\underbrace{y_2\sim t_2}},$$
and
$$P_i=\underset{G_{\tau(s_i)}}{\underbrace{s_i\sim v_i}}\underset{G_{\tau(t_i)}}{\underbrace{u_i\sim t_i}}$$
for $3\leq i\leq n$. Then $P_1,P_2,\dotsc,P_n$ form an $n$-PDPC of the weld of $G_1,G_2,G_3,G_4$, and we obtain a desired $n$-PDPC of $G$ by Lemma~\ref{lem:emptylayers}.

It remains to consider that $\overline{P}_1=u_1\sim x_0v_ns_ny_0\sim t_1$, where $x_0$ is black, $v_n$ and $y_0$ are white, and $x_0,y_0$ are not required to be distinct from $u_1,t_1$, respectively. By condition~$(\ref{item:cond3})$ on $G_2$, there is an $(n-1)$-PDPC $\widehat{P}_1,\widehat{P}_2,\dotsc,\widehat{P}_{n-1}$ in $G_2$ with endpoints $S_2\cup\{u_i:t_i\in T_2\setminus\{t_n\}\}$ and $\{v_i:s_i\in S_2\}\cup(T_2\setminus\{t_n\})$, where paths are of the form $s_i\sim v_i$ or $u_i\sim t_i$. Let $\widehat{P}_\iota$ contain $t_n$ for some $\iota\in[n-1]$. Let $\widehat{P}_\iota$ be either $s_\iota\sim x_\iota t_nu_ny_\iota\sim v_\iota$ or $u_\iota\sim x_\iota t_nu_ny_\iota\sim t_\iota$, where $u_n$ and $x_\iota$ are black, $y_\iota$ is white, $x_\iota$ is not required to be distinct from $s_\iota$ or $u_\iota$, and $y_\iota$ is not required to be distinct from $v_\iota$ or $t_\iota$. Let $v_n^*,x_0^*,y_0^*\in V(G_3)$ and $u_n^*,x_\iota^*,y_\iota^*\in V(G_4)$ be the unique neighbors of $v_n,x_0,y_0,u_n,x_\iota,y_\iota$, respectively. Let $y_n$ be a white vertex in $V(G_3)\setminus\big(\{x_0^*\}\cup N(y_\iota^*)\big)$, and let $x_n\in V(G_4)$ be the unique neighbor of $y_n$. By condition~$(\ref{item:cond3})$ on $G_3$ and $G_4$ together with Proposition~\ref{prop:differentk}$(\ref{item:ours})$, there is a $2$-PDPC in $G_3$ with endpoints $\{v_n^*,y_0^*\}$ and $\{y_n,x_0^*\}$ and a $2$-PDPC in $G_4$ with endpoints $\{x_n,y_\iota^*\}$ and $\{u_n^*,x_\iota^*\}$. If $\iota=1$, then let
$$P_1=\underset{\widehat{P}_1}{\underbrace{s_1\sim x_1}}\underset{G_4}{\underbrace{x_1^*\sim y_1^*}}\underset{\widehat{P}_1}{\underbrace{y_1\sim v_1}}\underset{\overline{P}_1}{\underbrace{u_1\sim x_0}}\underset{G_3}{\underbrace{x_0^*\sim y_0^*}}\underset{\overline{P}_1}{\underbrace{y_0\sim t_1}},$$
$$P_i=\underset{G_{\tau(s_i)}}{\underbrace{s_i\sim v_i}}\underset{G_{\tau(t_i)}}{\underbrace{u_i\sim t_i}}$$
for $2\leq i\leq n-1$, and 
$$P_n=\underset{\overline{P}_1}{\underbrace{s_nv_n}}\underset{G_3}{\underbrace{v_n^*\sim y_n}}\underset{G_4}{\underbrace{x_n\sim u_n^*}}\underset{\widehat{P}_1}{\underbrace{u_nt_n}};$$
if $\iota\neq1$, then assuming without loss of generality that $\iota=2$, let
$$P_1=\underset{\widehat{P}_1}{\underbrace{s_1\sim v_1}}\underset{\overline{P}_1}{\underbrace{u_1\sim x_0}}\underset{G_3}{\underbrace{x_0^*\sim y_0^*}}\underset{\overline{P}_1}{\underbrace{y_0\sim t_1}},$$
$P_2$ be either
$$\underset{\widehat{P}_2}{\underbrace{s_2\sim x_2}}\underset{G_4}{\underbrace{x_2^*\sim y_2^*}}\underset{\widehat{P}_2}{\underbrace{y_2\sim v_2}}\underset{\overline{P}_2}{\underbrace{u_2\sim t_2}}\quad\text{or}\quad\underset{\overline{P}_2}{\underbrace{s_2\sim v_2}}\underset{\widehat{P}_2}{\underbrace{u_2\sim x_2}}\underset{G_4}{\underbrace{x_2^*\sim y_2^*}}\underset{\widehat{P}_2}{\underbrace{y_2\sim t_2}},$$
$$P_i=\underset{G_{\tau(s_i)}}{\underbrace{s_i\sim v_i}}\underset{G_{\tau(t_i)}}{\underbrace{u_i\sim t_i}}$$
for $3\leq i\leq n-1$, and 
$$P_n=\underset{\overline{P}_1}{\underbrace{s_nv_n}}\underset{G_3}{\underbrace{v_n^*\sim y_n}}\underset{G_4}{\underbrace{x_n\sim u_n^*}}\underset{\widehat{P}_2}{\underbrace{u_nt_n}}.$$
Then $P_1,P_2,\dotsc,P_n$ form an $n$-PDPC of the weld of $G_1,G_2,G_3,G_4$, and we obtain a desired $n$-PDPC of $G$ by Lemma~\ref{lem:emptylayers}.\qedhere
\end{enumerate}
\end{proof}

%Instead of proceeding in accordance to the flow of logic, we first present the proof of our main theorem by assuming the results of Proposition~\ref{prop:inductionDPC}.
Before we prove our main theorem, we provide an equivalent form of Theorem~\ref{thm:transpositionDPC}.
\theoremstyle{plain}
\newtheorem*{theoremmain}{Restatement of Theorem~\ref{thm:transpositionDPC}}
\begin{theoremmain}
Let $n$ be a positive integer, and let $G$ be a bipartite transposition-like graph of rank $n+1$ with partite sets $V_1$ and $V_2$. Assume that during the welding process to form $G$, every rank $1$ transposition-like graph that $G$ is built up from is either a single vertex or has an even number of vertices. Then for every choice of $n$-subsets $S\subseteq V_1$ and $T\subseteq V_2$, $G$ admits an $n$-PDPC.
\end{theoremmain}

\begin{proof}[Proof of Theorem~$\ref{thm:transpositionDPC}$]
%Let $G$ be a bipartite transposition-like graph satisfying the conditions given in the statement of the theorem. Further 
Let $G$ be a weld of $G_1,G_2,\dotsc,G_\ell$, where $\ell\geq n+1$ and each $G_i$ is a bipartite transposition-like graph of rank $n$ with the same number of vertices. Note that if $n=3$, then the number of vertices in each $G_j$ needs to be at least $4n-2=10$ in order to satisfy condition~$(\ref{item:cond1})$ in Proposition~\ref{prop:inductionDPC}, and if $n\geq4$, then the number of vertices in each $G_j$ will be at least $n!=(n-1)!n\geq6n>4n-2$. Hence, in view of Proposition~\ref{prop:inductionDPC}, it suffices to show the following statements.
\begin{enumerate}[$(a)$]
\item\label{item:baserank2} If $G$ is of rank $2$, then $G$ is Hamiltonian-laceable.
\item\label{item:baserank3} If $G$ is of rank $3$, then for every choice of $2$-subsets $S\subseteq V_1$ and $T\subseteq V_2$, $G$ admits a $2$-PDPC.
\item\label{item:baserank4} If $G$ is of rank $4$ and every rank $3$ transposition-like graph that $G$ is built up from has at most $8$ vertices, then for every choice of $3$-subsets $S\subseteq V_1$ and $T\subseteq V_2$, $G$ admits a $3$-PDPC.%If $G=\Gamma(S_4,\mathcal{T}_4)$ is the transposition graph of rank $4$, then for every choice of $3$-subsets $S\subseteq V_1$ and $T\subseteq V_2$, $G$ admits a $3$-PDPC, where each path in this $3$-PDPC contains at least four vertices.
\end{enumerate}

We begin by proving $(\ref{item:baserank2})$.

\begin{enumerate}[\textit{Case} $1$:]
\item $|V(G_j)|=1$ for all $j\in[\ell]$.

Note that $G$ is the complete graph $K_\ell$, which is bipartite if and only if $\ell=2$. Hence, $G=K_2$ is Hamiltonian-laceable.

\item $|V(G_j)|$ is even for all $j\in[\ell]$.

Let $S=\{s\}$ and $T=\{t\}$. If $s,t\in V(G_j)$ for some $j\in[\ell]$, then since $G_j$ is a bipartite rank $1$ transposition-like graph, by Definition~\ref{def:transpositionlike}, there exists a Hamiltonian path in $G_j$ with endpoints $s$ and $t$, and we obtain a desired Hamiltonian path of $G$ by Lemma~\ref{lem:emptylayers}. Otherwise, assume without loss of generality that $s\in V(G_1)$ and $t\in V(G_2)$. Let $v\in V(G_1)$ be a white vertex and let $u\in V(G_2)$ be the unique neighbor of $v$. By Definition~\ref{def:transpositionlike} on $G_1$ and $G_2$, there is a Hamiltonian path in $G_1$ with endpoints $s$ and $v$ and a Hamiltonian path in $G_2$ with endpoints $u$ and $t$. Then $s\sim vu\sim t$ is a Hamiltonian path of the weld of $G_1$ and $G_2$, and we obtain a desired Hamiltonian path of $G$ by Lemma~\ref{lem:emptylayers}.
\end{enumerate}

Next, we provide the proof of $(\ref{item:baserank3})$.%To prove $(\ref{item:baserank3})$, let $G$ be a weld of $G_1,G_2,\dotsc,G_\ell$, where $\ell\geq3$ and each $G_i$ is a rank $2$ transposition-like graph with the same number of vertices. By $(\ref{item:baserank2})$, each $G_i$ is a Hamiltonian-laceable graph.

\begin{enumerate}[\textit{Case} $1$:]
\item $w_j\leq1$ for all $j\in[\ell]$.

For each $i\in\{1,2\}$ such that $\tau(s_i)\neq\tau(t_i)$, let $j=\tau(s_i)$ and $j^*=\tau(t_i)$. Let $v_i$ be a white vertex in $V(G_j)$ and let $u_i\in V(G_{j^*})$ be the unique neighbor of $v_i$. By the result in part~$(\ref{item:baserank2})$ on each $G_j$, there is a Hamiltonian path in $G_j$ with endpoints $S_j'\cup\{u_i:t_i\in T_j'\}\cup S_j''$ and $\{v_i:s_i\in S_j'\}\cup T_j'\cup T_j''$, where paths are of the form $s_i\sim v_i$, $u_i\sim t_i$, or $s_i\sim t_i$. For each $i\in\{1,2\}$, let
$$P_i=\underset{G_{\tau(s_i)}}{\underbrace{s_i\sim v_i}}\underset{G_{\tau(t_i)}}{\underbrace{u_i\sim t_i}}$$
if $s_i\in S_{\tau(s_i)}'$ and let
$$P_i=\underset{G_{\tau(s_i)}}{\underbrace{s_i\sim t_i}}$$
if $s_i\in S_{\tau(s_i)}''$. Then $P_1,P_2$ form a $2$-PDPC of the weld of the graphs in $\{G_j:j\in[\ell]\text{ and }w_j>0\}$, and we obtain a desired $2$-PDPC of $G$ by Lemma~\ref{lem:emptylayers}.

\item There exists $j\in[\ell]$ such that $S\cup T\subseteq V(G_j)$.

Without loss of generality, assume that $S\cup T\subseteq V(G_1)$. By the result in part~$(\ref{item:baserank2})$ on $G_1$, there is a Hamiltonian path $\overline{P}_1$ in $G_1$ with endpoints $s_1$ and $t_1$. Then $\overline{P}_1$ is either $s_1\sim v_1s_2\sim t_2u_1\sim t_1$ or $s_1\sim u_1t_2\sim s_2v_1\sim t_1$, where $u_1$ is black, $v_1$ is white, and $u_1,v_1$ are not required to be distinct from $s_1,t_1$, respectively. Let $u_1^*,v_1^*\in V(G_2)$ be the unique neighbors of $u_1,v_1$, respectively. By the result in part~$(\ref{item:baserank2})$ on $G_2$, there is a Hamiltonian path in $G_2$ with endpoints $v_1^*$ and $u_1^*$. Let $P_1$ be either
$$\underset{\overline{P}_1}{\underbrace{s_1\sim v_1}}\underset{G_2}{\underbrace{v_1^*\sim u_1^*}}\underset{\overline{P}_1}{\underbrace{u_1\sim t_1}}\quad\text{or}\quad\underset{\overline{P}_1}{\underbrace{s_1\sim u_1}}\underset{G_2}{\underbrace{u_1^*\sim v_1^*}}\underset{\overline{P}_1}{\underbrace{v_1\sim t_1}}$$
and
$$P_2=\underset{\overline{P}_1}{\underbrace{s_2\sim t_2}}.$$
Then $P_1,P_2$ form a $2$-PDPC of the weld of $G_1$ and $G_2$, and we obtain a desired $2$-PDPC of $G$ by Lemma~\ref{lem:emptylayers}.

\item There exist $j_1,j_2\in[\ell]$, where $j_1\neq j_2$, such that $S\subseteq V(G_{j_1})$ and $T\subseteq V(G_{j_2})$.

Without loss of generality, assume that $S\subseteq V(G_1)$ and $T\subseteq V(G_2)$. Let $v_1\in V(G_1)$ be a white vertex. By the result in part~$(\ref{item:baserank2})$ on $G_1$, there is a Hamiltonian path $\overline{P}_1$ in $G_1$ with endpoints $s_1$ and $v_1$. Then $\overline{P}_1=s_1\sim v_2s_2\sim v_1$, where $v_2$ is white. Rename the vertices $v_1$ and $v_2$ as $v_2$ and $v_1$, respectively, so $\overline{P}_1=s_1\sim v_1s_2\sim v_2$. Let $v_1^*,v_2^*\in V(G_2)$ be the unique neighbors of $v_1,v_2$, respectively. By the result in part~$(\ref{item:baserank2})$ on $G_2$, there is a Hamiltonian path $\widehat{P}_1$ in $G_2$ with endpoints $v_1^*$ and $t_1$. Then $\widehat{P}_1$ is either $v_1^*\sim v_0v_2^*\sim t_2u_0\sim t_1$ or $v_1^*\sim u_0t_2\sim v_2^*v_0\sim t_1$, where $u_0$ is black, $v_0$ is white, and $u_0,v_0$ are not required to be distinct from $v_1^*,t_1$, respectively. Let $u_0^*,v_0^*\in V(G_3)$ be the unique neighbors of $u_0,v_0$, respectively. By the result in part~$(\ref{item:baserank2})$ on $G_3$, there is a Hamiltonian path in $G_3$ with endpoints $v_0^*$ and $u_0^*$. Let $P_1$ be either
$$\underset{\overline{P}_1}{\underbrace{s_1\sim v_1}}\underset{\widehat{P}_1}{\underbrace{v_1^*\sim v_0}}\underset{G_3}{\underbrace{v_0^*\sim u_0^*}}\underset{\widehat{P}_1}{\underbrace{u_0\sim t_1}}\quad\text{or}\quad\underset{\overline{P}_1}{\underbrace{s_1\sim v_1}}\underset{\widehat{P}_1}{\underbrace{v_1^*\sim u_0}}\underset{G_3}{\underbrace{u_0^*\sim v_0^*}}\underset{\widehat{P}_1}{\underbrace{v_0\sim t_1}}$$
and
$$P_2=\underset{\overline{P}_1}{\underbrace{s_2\sim v_2}}\underset{\widehat{P}_1}{\underbrace{v_2^*\sim t_2}}.$$
Then $P_1,P_2$ form a $2$-PDPC of the weld of $G_1,G_2,G_3$, and we obtain a desired $2$-PDPC of $G$ by Lemma~\ref{lem:emptylayers}.

\item There exists $j_0\in[\ell]$ such that $S\subseteq V(G_{j_0})$ or $T\subseteq V(G_{j_0})$ but not both, and $w_j\leq1$ for all $j\in[\ell]\setminus\{j_0\}$.

Without loss of generality, assume that $S\subseteq V(G_1)$ and $t_2\notin V(G_1)$. If $t_1\notin V(G_1)$, then without loss of generality, assume that $t_1\in V(G_2)$ and $t_2\in V(G_3)$. Let $v_1\in V(G_1)$ be a white vertex. By the result in part~$(\ref{item:baserank2})$ on $G_1$, there is a Hamiltonian path $\overline{P}_1$ in $G_1$ with endpoints $s_1$ and $v_1$. Then $\overline{P}_1=s_1\sim v_2s_2\sim v_1$, where $v_2$ is white. Rename the vertices $v_1$ and $v_2$ as $v_2$ and $v_1$, respectively, so $\overline{P}_1=s_1\sim v_1s_2\sim v_2$. Let $v_1^*\in V(G_2)$ and $v_2^*\in V(G_3)$ be the unique neighbors of $v_1,v_2$, respectively. By the result in part~$(\ref{item:baserank2})$ on $G_2$ and $G_3$, there is a Hamiltonian path in $G_2$ with endpoints $v_1^*$ and $t_1$ and a Hamiltonian path in $G_3$ with endpoints $v_2^*$ and $t_2$.  Let
$$P_1=\underset{\overline{P}_1}{\underbrace{s_1\sim v_1}}\underset{G_2}{\underbrace{v_1^*\sim t_1}}$$
and
$$P_2=\underset{\overline{P}_1}{\underbrace{s_2\sim v_2}}\underset{G_3}{\underbrace{v_2^*\sim t_2}}.$$
Then $P_1,P_2$ form a $2$-PDPC of the weld of $G_1,G_2,G_3$, and we obtain a desired $2$-PDPC of $G$ by Lemma~\ref{lem:emptylayers}.

If $t_1\in V(G_1)$, then without loss of generality, assume that $t_2\in V(G_2)$. By the result in part~$(\ref{item:baserank2})$ on $G_1$, there is a Hamiltonian path $\overline{P}_1$ in $G_1$ with endpoints $s_1$ and $t_1$. Let $\overline{P}_1=s_1\sim u_1v_2s_2v_1\sim t_1$, where $u_1$ is black, $v_1$ is white, and $u_1,v_1$ are not required to be distinct from $s_1,t_1$, respectively. Let $u_1^*,v_1^*\in V(G_3)$ and $v_2^*\in V(G_2)$ be the unique neighbors of $u_1,v_1,v_2$, respectively. By the result in part~$(\ref{item:baserank2})$ on $G_2$ and $G_3$, there is a Hamiltonian path in $G_2$ with endpoints $v_2^*$ and $t_2$ and a Hamiltonian path in $G_3$ with endpoints $v_1^*$ and $u_1^*$.  Let
$$P_1=\underset{\overline{P}_1}{\underbrace{s_1\sim u_1}}\underset{G_3}{\underbrace{u_1^*\sim v_1^*}}\underset{\overline{P}_1}{\underbrace{v_1\sim t_1}}$$
and 
$$P_2=\underset{\overline{P}_1}{\underbrace{s_2v_2}}\underset{G_2}{\underbrace{v_2^*\sim t_2}}.$$
Then $P_1,P_2$ form a $2$-PDPC of the weld of $G_1,G_2,G_3$, and we obtain a desired $2$-PDPC of $G$ by Lemma~\ref{lem:emptylayers}.

\item There exists $j_0\in[\ell]$ such that $w_{j_0}=2$, $S\not\subseteq V(G_{j_0})$, $T\not\subseteq V(G_{j_0})$, and $w_j\leq1$ for all $j\in[\ell]\setminus\{j_0\}$.

Without loss of generality, assume that $\{s_1,t_2\}\subseteq V(G_1)$, $t_1\in V(G_2)$, and $s_2\in V(G_3)$. If $|V(G_j)|=2$ for all $j\in[\ell]$, then let $V(G_2)=\{u,t_1\}$ and $V(G_3)=\{s_2,v\}$, where $u$ is black and $v$ is white. Let $P_1=s_1vut_1$ and $P_2=s_2t_2$. Then $P_1,P_2$ form a $2$-PDPC of the weld of $G_1,G_2,G_3$, and we obtain a desired $2$-PDPC of $G$ by Lemma~\ref{lem:emptylayers}.

If $|V(G_j)|>2$ for all $j\in[\ell]$, then by the result in part~$(\ref{item:baserank2})$ on $G_1$, there is a Hamiltonian path $\overline{P}_1$ in $G_1$ with endpoints $s_1$ and $t_2$. Let $\overline{P}_1=s_1\sim vu\sim t_2$, where $u$ is black and $v$ is white. Let $u^*\in V(G_3)$ and $v^*\in V(G_2)$ be the unique neighbors of $u$ and $v$, respectively. By the result in part~$(\ref{item:baserank2})$ on $G_2$ and $G_3$, there is a Hamiltonian path in $G_2$ with endpoints $v^*$ and $t_1$ and a Hamiltonian path in $G_3$ with endpoints $s_2$ and $u^*$.  Let 
$$P_1=\underset{\overline{P}_1}{\underbrace{s_1\sim v}}\underset{G_2}{\underbrace{v^*\sim t_1}}$$
and 
$$P_2=\underset{G_3}{\underbrace{s_2\sim u^*}}\underset{\overline{P}_1}{\underbrace{u\sim t_2}}.$$ Then $P_1,P_2$ form a $2$-PDPC of the weld of $G_1,G_2,G_3$, and we obtain a desired $2$-PDPC of $G$ by Lemma~\ref{lem:emptylayers}.

\item There exist $j_1,j_2\in[\ell]$, where $j_1\neq j_2$, such that $w_{j_1}=w_{j_2}=2$ and $S\not\subseteq V(G_{j_1})$.

Without loss of generality, assume that $\{s_1,t_2\}\subseteq V(G_1)$ and $\{s_2,t_1\}\subseteq V(G_2)$. Let $u_1\in V(G_1)$ be the unique neighbor of $t_1$, where $u_1$ is not required to be distinct from $s_1$. By the result in part~$(\ref{item:baserank2})$ on $G_1$ and $G_2$, there is a Hamiltonian path $\overline{P}_1$ in $G_1$ with endpoints $s_1$ and $t_2$ and a Hamiltonian path $\overline{P}_2$ in $G_2$ with endpoints $s_2$ and $t_1$. Let $\overline{P}_1=s_1\sim u_1v_2\sim t_2$ and $\overline{P}_2=s_2\sim u_2t_1$, where $u_2$ is black, $v_2$ is white, and $u_2,v_2$ are not required to be distinct from $s_2,t_2$, respectively. Let $u_2^*,v_2^*\in V(G_3)$ be the unique neighbors of $u_2,v_2$, respectively. By the result in part~$(\ref{item:baserank2})$ on $G_3$, there is a Hamiltonian path in $G_3$ with endpoints $v_2^*$ and $u_2^*$. Let
$$P_1=\underset{\overline{P}_1}{\underbrace{s_1\sim u_1}}\underset{\overline{P}_2}{\underbrace{t_1}}$$ and
$$P_2=\underset{\overline{P}_2}{\underbrace{s_2\sim u_2}}\underset{G_3}{\underbrace{u_2^*\sim v_2^*}}\underset{\overline{P}_1}{\underbrace{v_2\sim t_2}}.$$
Then $P_1,P_2$ form a $2$-PDPC of the weld of $G_1,G_2,G_3$, and we obtain a desired $2$-PDPC of $G$ by Lemma~\ref{lem:emptylayers}.
\end{enumerate}

Lastly, we prove $(\ref{item:baserank4})$. Since $G_j$ is a bipartite transposition-like graph of rank $3$ for all $j\in[\ell]$, there are at least $3!=6$ vertices in $V(G_j)$. By Lemma~\ref{lem:bipartiteequitable}, there are at least $3$ black and $3$ white vertices in $V(G_j)$.

\begin{enumerate}[\textit{Case} $1$:]
\item $w_j\leq2$ for all $j\in[\ell]$.

The proof of Case~$\ref{item:wj<=n-1}$ in Proposition~\ref{prop:inductionDPC} applies due to $2(3-1)-2=2<3$ and the result in part~$(\ref{item:baserank3})$ on each $G_j$.

\item There exists $j\in[\ell]$ such that $S\cup T\subseteq V(G_j)$.

The proof of Case~$\ref{item:STinVGj}$ in Proposition~\ref{prop:inductionDPC} applies due to the result in part~$(\ref{item:baserank3})$ on each $G_j$.

\item There exist $j_1,j_2\in[\ell]$, where $j_1\neq j_2$, such that $S\subseteq V(G_{j_1})$ and $T\subseteq V(G_{j_2})$.

The proof of Case~$\ref{item:SinVG1TinVG2}$ in Proposition~\ref{prop:inductionDPC} applies due to the result in part~$(\ref{item:baserank3})$ on each $G_j$.

\item There exists $j_0\in[\ell]$ such that $S\subseteq V(G_{j_0})$ or $T\subseteq V(G_{j_0})$ but not both, and $w_j\leq2$ for all $j\in[\ell]\setminus\{j_0\}$.

The proof of Case~$\ref{item:SinVGj0TnotinVGj0}$ in Proposition~\ref{prop:inductionDPC} applies due to the result in part~$(\ref{item:baserank3})$ on each $G_j$.

\item There exists $j_0\in[\ell]$ such that $w_{j_0}=3$, $S\nsubseteq V(G_{j_0})$, $T\nsubseteq V(G_{j_0})$, and $w_j\leq2$ for all $j\in[\ell]\setminus\{j_0\}$.

Without loss of generality, assume that $\{s_1,s_2,t_3\}\subseteq V(G_1)$.

\begin{enumerate}[\textit{Case $5.$}$1$:]
\item $\{t_1,t_2\}\cap V(G_1)\neq\emptyset$.

Without loss of generality, assume that $t_1\in V(G_1)$ and $t_2\in V(G_2)$. By the result in part~$(\ref{item:baserank3})$ on $G_1$, there is a $2$-PDPC $\overline{P}_1,\overline{P}_2$ in $G_1$ with endpoints $\{s_1,s_2\}$ and $\{t_1,t_3\}$. Let $\overline{P}_2=s_2v_3\sim t_3$, where $v_3$ is white and $v_3$ is not required to be distinct from $t_3$. If $s_3\in V(G_2)$, then let $v_3^*\in V(G_2)$ be the unique neighbor of $v_3$, where $v_3^*$ is not required to be distinct from $s_3$. By the result in part~$(\ref{item:baserank3})$ on $G_2$ together with Proposition~\ref{prop:differentk}$(\ref{item:ours})$, there is a Hamiltonian path $\widehat{P}_2$ of $G_2$ with endpoints $s_3$ and $t_2$. Let $\widehat{P}_2=s_3\sim v_3^*v_2\sim t_2$, where $v_2$ is white and $v_2$ is not required to be distinct from $t_2$. Let $u_2^*,v_2^*\in V(G_3)$ be the unique neighbors of $s_2,v_2$, respectively. By the result in part~$(\ref{item:baserank3})$ on $G_3$ together with Proposition~\ref{prop:differentk}$(\ref{item:ours})$, there is a Hamiltonian path of $G_3$ with endpoints $v_2^*$ and $u_2^*$. Let $P_1=\overline{P}_1$, 
$$P_2=\underset{\overline{P}_2}{\underbrace{s_2}}\underset{G_3}{\underbrace{u_2^*\sim v_2^*}}\underset{\widehat{P}_2}{\underbrace{v_2\sim t_2}},$$
and
$$P_3=\underset{\widehat{P}_2}{\underbrace{s_3\sim v_3^*}}\underset{\overline{P}_2}{\underbrace{v_3\sim t_3}}.$$
Then $P_1,P_2,P_3$ form a $3$-PDPC of the weld of $G_1,G_2,G_3$, and we obtain a desired $3$-PDPC of $G$ by Lemma~\ref{lem:emptylayers}.

If $s_3\notin V(G_2)$, then assume without loss of generality that $s_3\in V(G_3)$. Let $v_2,v_3^*\in V(G_2)$ be the unique neighbors of $s_2,v_3$, respectively, where $v_2$ is not required to be distinct from $t_2$. By the result in part~$(\ref{item:baserank3})$ on $G_2$ together with Proposition~\ref{prop:differentk}$(\ref{item:ours})$, there is a Hamiltonian path $\widehat{P}_2$ of $G_2$ with endpoints $v_3^*$ and $t_2$. Let $\widehat{P}_2=v_3^*\sim u_3v_2\sim t_2$, where $u_3$ is black and $u_3$ is not required to be distinct from $v_3^*$. Let $u_3^*\in V(G_3)$ be the unique neighbor of $u_3$. By the result in part~$(\ref{item:baserank3})$ on $G_3$ together with Proposition~\ref{prop:differentk}$(\ref{item:ours})$, there is a Hamiltonian path of $G_3$ with endpoints $s_3$ and $u_3^*$. Let $P_1=\overline{P}_1$,
$$P_2=\underset{\overline{P}_2}{\underbrace{s_2}}\underset{\widehat{P}_2}{\underbrace{v_2\sim t_2}},$$
and
$$P_3=\underset{G_3}{\underbrace{s_3\sim u_3^*}}\underset{\widehat{P}_2}{\underbrace{u_3\sim v_3^*}}\underset{\overline{P}_2}{\underbrace{v_3\sim t_3}}.$$
Then $P_1,P_2,P_3$ form a $3$-PDPC of the weld of $G_1,G_2,G_3$, and we obtain a desired $3$-PDPC of $G$ by Lemma~\ref{lem:emptylayers}.

\item $\{t_1,t_2\}\cap V(G_1)=\emptyset$.

Without loss of generality, assume that $s_3\in V(G_2)$. If $\{t_1,t_2\}\cap V(G_2)\neq\emptyset$, then assume without loss of generality that $t_1\in V(G_2)$ and $t_2\in V(G_3)$. Let $v_1\in V(G_1)\setminus(\{t_3\}\cup N(s_3))$ be a white vertex and let $u_1\in V(G_2)$ be the unique neighbor of $v_1$. By the result in part~$(\ref{item:baserank3})$ on $G_1$, there is a $2$-PDPC $\overline{P}_1,\overline{P}_2$ in $G_1$ with endpoints $\{s_1,s_2\}$ and $\{v_1,t_3\}$. Let $\overline{P}_2=s_2v_3\sim t_3$, where $v_3$ is white and $v_3$ is not required to be distinct from $t_3$. Let $v_3^*\in V(G_3)$ be the unique neighbor of $v_3$. Let $u_3\in V(G_3)\setminus(\{v_3^*\}\cup N(t_1))$ be a black vertex and let $u_3^*\in V(G_2)$ be the unique neighbor of $u_3$. By the result in part~$(\ref{item:baserank3})$ on $G_2$, there is a $2$-PDPC in $G_2$ with endpoints $\{u_1,s_3\}$ and $\{t_1,u_3^*\}$. Furthermore, by the result in part~$(\ref{item:baserank3})$ on $G_3$ together with Proposition~\ref{prop:differentk}$(\ref{item:ours})$, there is a Hamiltonian path $\widehat{P}_2$ of $G_3$ with endpoints $v_3^*$ and $t_2$. Let $\widehat{P}_2=v_3^*\sim u_3v_2\sim t_2$, where $v_2$ is white and $v_2$ is not required to be distinct from $t_2$. Let $u_2^*,v_2^*\in V(G_4)$ be the unique neighbors of $s_2,v_2$, respectively. By the result in part~$(\ref{item:baserank3})$ on $G_4$ together with Proposition~\ref{prop:differentk}$(\ref{item:ours})$, there is a Hamiltonian path of $G_4$ with endpoints $v_2^*$ and $u_2^*$. Let 
$$P_1=\underset{\overline{P}_1}{\underbrace{s_1\sim v_1}}\underset{G_2}{\underbrace{u_1\sim t_1}},$$ 
$$P_2=\underset{\overline{P}_2}{\underbrace{s_2}}\underset{G_4}{\underbrace{u_2^*\sim v_2^*}}\underset{\widehat{P}_2}{\underbrace{v_2\sim t_2}},$$
and
$$P_3=\underset{G_2}{\underbrace{s_3\sim u_3^*}}\underset{\widehat{P}_2}{\underbrace{u_3\sim v_3^*}}\underset{\overline{P}_2}{\underbrace{v_3\sim t_3}}.$$
Then $P_1,P_2,P_3$ form a $3$-PDPC of the weld of $G_1,G_2,G_3,G_4$, and we obtain a desired $3$-PDPC of $G$ by Lemma~\ref{lem:emptylayers}.

If $\{t_1,t_2\}\cap V(G_2)=\emptyset$, then without loss of generality, we have either $\{t_1,t_2\}\subseteq V(G_3)$, or $t_1\in V(G_3)$ and $t_2\in V(G_4)$. If $\{t_1,t_2\}\subseteq V(G_3)$, then let $v_1\in V(G_1)\setminus\{t_3\}$ be a white vertex, and let $u_1\in V(G_3)$ be the unique neighbor of $v_1$. By the result in part~$(\ref{item:baserank3})$ on $G_1$, there is a $2$-PDPC $\overline{P}_1,\overline{P}_2$ in $G_1$ with endpoints $\{s_1,s_2\}$ and $\{v_1,t_3\}$. Let $\overline{P}_2=s_2v_3\sim t_3$, where $v_3$ is white and $v_3$ is not required to be distinct from $t_3$. Let $v_3^*\in V(G_3)$ be the unique neighbor of $v_3$. By the result in part~$(\ref{item:baserank3})$ on $G_3$, there is a $2$-PDPC $\widehat{P}_1,\widehat{P}_2$ in $G_3$ with endpoints $\{u_1,v_3^*\}$ and $\{t_1,t_2\}$. Let $\widehat{P}_3=v_3^*\sim u_3t_2$, where $u_3$ is black and $u_3$ is not required to be distinct from $v_3^*$. Let $u_3^*\in V(G_2)$ and $v_2,u_2\in V(G_4)$ be the unique neighbors of $u_3,s_2,t_2$, respectively. By the result in part~$(\ref{item:baserank3})$ on $G_2$ and $G_4$ together with Proposition~\ref{prop:differentk}$(\ref{item:ours})$, there is a Hamiltonian path of $G_2$ with endpoints $s_3$ and $u_3^*$ and a Hamiltonian path of $G_4$ with endpoints $u_2$ and $v_2$. Let
$$P_1=\underset{\overline{P}_1}{\underbrace{s_1\sim v_1}}\underset{\widehat{P}_1}{\underbrace{u_1\sim t_1}},$$
$$P_2=\underset{\overline{P}_2}{\underbrace{s_2}}\underset{G_4}{\underbrace{v_2\sim u_2}}\underset{\widehat{P}_3}{\underbrace{t_2}},$$
and
$$P_3=\underset{G_2}{\underbrace{s_3\sim u_3^*}}\underset{\widehat{P}_3}{\underbrace{u_3\sim v_3^*}}\underset{\overline{P}_2}{\underbrace{v_3\sim t_3}}.$$
Then $P_1,P_2,P_3$ form a $3$-PDPC of the weld of $G_1,G_2,G_3,G_4$, and we obtain a desired $3$-PDPC of $G$ by Lemma~\ref{lem:emptylayers}.

If $t_1\in V(G_3)$ and $t_2\in V(G_4)$, then $v_1\in V(G_1)\setminus\{t_3\}$ be a white vertex, and let $u_1\in V(G_3)$ be the unique neighbor of $v_1$. By the result in part~$(\ref{item:baserank3})$ on $G_1$, there is a $2$-PDPC $\overline{P}_1,\overline{P}_2$ in $G_1$ with endpoints $\{s_1,s_2\}$ and $\{v_1,t_3\}$. Let $\overline{P}_2=s_2v_3\sim t_3$, where $v_3$ is white and $v_3$ is not required to be distinct from $t_3$. Let $v_2,v_3^*\in V(G_4)$ be the unique neighbor of $s_2,v_3$, respectively, where $v_2$ is not required to be distinct from $t_2$. By the result in part~$(\ref{item:baserank3})$ on $G_4$ together with Proposition~\ref{prop:differentk}$(\ref{item:ours})$, there is a Hamiltonian path $\widehat{P}_2$ in $G_4$ with endpoints $v_3^*$ and $t_2$. Let $\widehat{P}_2=v_3^*\sim u_3v_2\sim t_2$, where $u_3$ is black and $u_3$ is not required to be distinct from $v_3^*$. Let $u_3^*\in V(G_2)$ be the unique neighbor of $u_3$. By the result in part~$(\ref{item:baserank3})$ on $G_2$ and $G_3$ together with Proposition~\ref{prop:differentk}$(\ref{item:ours})$, there is a Hamiltonian path in $G_2$ with endpoints $s_3$ and $u_3^*$ and a Hamiltonian path in $G_3$ with endpoints $u_1$ and $t_1$. Let
$$P_1=\underset{\overline{P}_1}{\underbrace{s_1\sim v_1}}\underset{G_3}{\underbrace{u_1\sim t_1}},$$ $$P_2=\underset{\overline{P}_2}{\underbrace{s_2}}\underset{\widehat{P}_2}{\underbrace{v_2\sim t_2}},$$
and
$$P_3=\underset{G_2}{\underbrace{s_3\sim u_3^*}}\underset{\widehat{P}_2}{\underbrace{u_3\sim v_3^*}}\underset{\overline{P}_2}{\underbrace{v_3\sim t_3}}.$$
Then $P_1,P_2,P_3$ form a $3$-PDPC of the weld of $G_1,G_2,G_3,G_4$, and we obtain a desired $3$-PDPC of $G$ by Lemma~\ref{lem:emptylayers}.
\end{enumerate}

\item There exist $j_1,j_2\in[\ell]$, where $j_1\neq j_2$, such that $w_{j_1}=w_{j_2}=3$ and $S\nsubseteq V(G_{j_1})$.

Without loss of generality, assume that $\{s_1,s_2,t_3\}\subseteq V(G_1)$ and $\{s_3,t_1,t_2\}\subseteq V(G_2)$. Let $v_1\in V(G_1)\setminus(\{t_3\}\cup N(s_3))$ be a white vertex and let $u_1\in V(G_2)$ be the unique neighbor of $v_1$. By the result in part~$(\ref{item:baserank3})$ on $G_1$ and $G_2$, there is a $2$-PDPC $\overline{P}_1,\overline{P}_2$ in $G_1$ with endpoints $\{s_1,s_2\}$ and $\{v_1,t_3\}$ and a $2$-PDPC $\widehat{P}_1,\widehat{P}_2$ in $G_2$ with endpoints $\{u_1,s_3\}$ and $\{t_1,t_2\}$. Let $\overline{P_2}=s_2v_3\sim t_3$ and $\widehat{P}_2=s_3\sim u_3t_2$, where $u_3$ is black, $v_3$ is white, and $u_3,v_3$ are not required to be distinct from $s_3,t_3$, respectively. Let $v_2,u_2\in V(G_3)$ be the unique neighbors of $s_2,t_2$, respectively, and let $u_3^*,v_3^*\in V(G_4)$ be the unique neighbors of $u_3,v_3$, respectively. By the result in part~$(\ref{item:baserank3})$ on $G_3$ and $G_4$ together with Proposition~\ref{prop:differentk}$(\ref{item:ours})$, there is a Hamiltonian path of $G_3$ with endpoints $u_2$ and $v_2$ and a Hamiltonian path of $G_4$ with endpoints $v_3^*$ and $u_3^*$. Let
$$P_1=\underset{\overline{P}_1}{\underbrace{s_1\sim v_1}}\underset{\widehat{P}_1}{\underbrace{u_1\sim t_1}},$$
$$P_2=\underset{\overline{P}_2}{\underbrace{s_2}}\underset{G_3}{\underbrace{v_2\sim u_2}}\underset{\widehat{P}_2}{\underbrace{t_2}},$$
and
$$P_3=\underset{\widehat{P}_2}{\underbrace{s_3\sim u_3}}\underset{G_4}{\underbrace{u_3^*\sim v_3^*}}\underset{\overline{P}_2}{\underbrace{v_3\sim t_3}}.$$ Then $P_1,P_2,P_3$ form a $3$-PDPC of the weld of $G_1,G_2,G_3,G_4$, and we obtain a desired $3$-PDPC of $G$ by Lemma~\ref{lem:emptylayers}.\qedhere
\end{enumerate}
\end{proof}

\section{Acknowledgements}

These results are based on work supported by the National Science Foundation under grant numbered MPS-2150299.

\end{document}